\documentclass[12pt, reqno]{amsart}
\usepackage[T1]{fontenc}
\usepackage{amsfonts}
\usepackage{amsmath}
\usepackage{amsthm}
\usepackage{amssymb}
\usepackage{geometry}
\usepackage{graphicx}
\usepackage{xcolor}
\usepackage{mathtools}
\usepackage[colorlinks=true, linkcolor=blue]{hyperref}
\usepackage{cleveref}
\usepackage[english]{babel}
\usepackage{lineno}
\usepackage{float}

\newtheorem{theorem}{Theorem}[section]

\newtheorem{example}[theorem]{Example}
\newtheorem{corollary}[theorem]{Corollary}
\newtheorem{lemma}[theorem]{Lemma}

\newtheorem{proposition}[theorem]{Proposition}
\newtheorem{problem}[theorem]{Problem}

\usepackage{tikz}
\usetikzlibrary{positioning}
\usetikzlibrary{decorations,arrows}
\usetikzlibrary{decorations.markings}
\numberwithin{equation}{section}

\newcommand{\N}{\mathbb N}
\newcommand{\Z}{\mathbb Z}

\usepackage{rustic}
\usepackage[T1]{fontenc}

\emergencystretch=\maxdimen
\title{Discrete isoperimetric inequalities on the strong products of paths}
\author{Runze Wang}
\address[]{Department of Mathematical Sciences, University of Memphis, Memphis, TN 38152, USA}
\email{rwang6@memphis.edu; runze.w@hotmail.com}
\thanks{}
\date{\today}

\begin{document}

\sloppy

\begin{abstract}
    For a graph $G=(V,\ E)$ and a nonempty set $S\subseteq V$, the \emph{vertex boundary} of $S$, denoted by $\partial_G(S)$, is defined to be the set of vertices that are not in $S$ but have at least one neighbor in $S$. In this paper, for $G$ being a strong product of two paths, we determine the cases in which $|\partial_G(S)|$ is minimized.
\end{abstract}

\maketitle

\section{Introduction}

In this paper, we only consider simple graphs, that is, undirected graphs without loops or multiedges.

For a graph $G=(V,\ E)$ and a nonempty set $S\subseteq V$, let the \emph{vertex boundary} of $S$, denoted by $\partial_G(S)$, be the set of vertices that are not in $S$ but have at least one neighbor in $S$. So we have
\begin{align*}
    \partial_G(S)=\{v\in V\setminus S: N(v)\cap S\neq \emptyset\}.
\end{align*}

Given that $|S|=k$, how small can $|\partial_G(S)|$ be? This problem is known as one of the \emph{discrete isoperimetric problems} on graphs, and it has attracted significant attention from many scholars. Wang and Wang \cite{WW}, as well as Bollob\'as and Leader \cite{BL2}, studied this problem on the grid graphs. We will further introduce their results later. Bollob\'as and Leader \cite{BL1} also studied this problem on the discrete toruses. Bobkov \cite{Bo} studied it on the discrete cubes (hypercubes), Bharadwaj and Chandran \cite{BC} studied it on trees, and Iamphongsai and Kittipassorn \cite{IK} studied it on the triangular grid graphs. Moreover, this problem was studied on the Cartesian products of graphs by Houdr\'e and Stoyanov \cite{HS} and Houdr\'e and Tetali \cite{HT}. In this paper, we focus on the strong products of paths. The compression method we use is derived from Bollob\'as and Leader's paper \cite{BL2}.

We define the \emph{infinite path} $P_{\N_0}$ to be the graph with vertex set $\N_0$ where two vertices $x$ and $y$ are adjacent if and only if $|x-y|=1$, and similarly define the \emph{double-sided infinite path} $P_\Z$ to be the graph with vertex set $\Z$ where two vertices $x$ and $y$ are adjacent if and only if $|x-y|=1$. The \emph{Cartesian product} of two graphs $G$ and $H$, denoted by $G\square H$, is the graph with vertex set $\{(v_1,\ v_2):\ v_1\in G,\ v_2\in H\}$ where two vertices $(v_1,\ v_2)$ and $(v_1',\ v_2')$ are adjacent if and only if either
\begin{center}
    $v_1=v_1'$ and $v_2v_2'\in E(H)$
\end{center}
or
\begin{center}
    $v_1v_1'\in E(G)$ and $v_2=v_2'$.
\end{center}
The Cartesian product of $n$ infinite paths, denoted by $\square_{i=1}^n P_{\N_0}$, can also be seen as a graph defined by letting the vertex set be $\{(v_1,\ v_2,\ ...,\ v_n):\ v_i\in \N_0,\ \forall\ 1\le i\le n\}$, and letting two vertices $(v_1,\ v_2,\ ...,\ v_n)$ and $(v_1',\ v_2',\ ...,\ v_n')$ be adjacent if $|v_j-v_j'|=1$ for some $1\le j\le n$, and $v_\ell=v_\ell'$ for any $1\le \ell\le n$ with $\ell\neq j$. Visually, this graph is just an $n$-dimensional grid.

The \emph{simplicial order} on the vertices in $\square_{i=1}^n P_{\N_0}$ (or $\square_{i=1}^n P_m$ with $P_m$ being a finite path on $m$ vertices) is defined by setting $(v_1,\ v_2,\ ...,\ v_n)<(v_1',\ v_2',\ ...,\ v_n')$ if either 
\begin{align*}
    \sum_{i=1}^n v_i<\sum_{i=1}^n v_i'
\end{align*}
or
\begin{center}
    $\sum_{i=1}^n v_i=\sum_{i=1}^n v_i'$; $v_j>v_j'$ for some $1\le j\le n$; $v_\ell=v_\ell'$ for any $1\le \ell\le j-1$.
\end{center}


The following result was proved by Wang and Wang \cite{WW} with an inductive argument on $n$ in 1977, and it was also proved by Bollob\'as and Leader \cite{BL2} using a compression method in 1991.

\begin{theorem}[Wang and Wang \cite{WW}, Bollob\'as and Leader \cite{BL2}]\label{thm1}
    Let $S$ be a set of $k$ vertices in $\square_{i=1}^n P_{\N_0}$. Then $|\partial_{\square_{i=1}^n P_{\N_0}}(S)|\ge |\partial_{\square_{i=1}^n P_{\N_0}}(S_0)|$, where $S_0$ is the set of the first $k$ vertices in the simplicial order.
\end{theorem}

Furthermore, Bollob\'as and Leader \cite{BL2} extended this result to the Cartesian product of finite paths of equal length.

\begin{theorem}[Bollob\'as and Leader \cite{BL2}]\label{thm2}
    Let $S$ be a set of $k\le m^n$ vertices in $\square_{i=1}^n P_m$. Then $|\partial_{\square_{i=1}^n P_m}(S)|\ge |\partial_{\square_{i=1}^n P_m}(S_0)|$, where $S_0$ is the set of the first $k$ vertices in the simplicial order.
\end{theorem}

In this paper, instead of the Cartesian products, we study the \emph{strong products}. The strong product of two graphs $G$ and $H$, denoted by $G\boxtimes H$, is the graph with vertex set $\{(v_1,\ v_2):\ v_1\in G,\ v_2\in H\}$ where two vertices $(v_1,\ v_2)$ and $(v_1',\ v_2')$ are adjacent if and only if 
\begin{center}
    $v_1=v_1'$ and $v_2v_2'\in E(H)$
\end{center}
or
\begin{center}
    $v_1v_1'\in E(G)$ and $v_2=v_2'$
\end{center}
or
\begin{center}
    $v_1v_1'\in E(G)$ and $v_2v_2'\in E(H)$.
\end{center}
The strong product of $n$ infinite paths, denoted by $\boxtimes_{i=1}^n P_{\N_0}$, can also be seen as a graph defined by letting the vertex set be $\{(v_1,\ v_2,\ ...,\ v_n):\ v_i\in \N_0,\ \forall\ 1\le i\le n\}$, and letting two vertices $(v_1,\ v_2,\ ...,\ v_n)\neq (v_1',\ v_2',\ ...,\ v_n')$ be adjacent if $|v_j-v_j'|\le 1$ for any $1\le j\le n$.

We define the \emph{box order} on the vertices in $\boxtimes_{i=1}^n P_{\N_0}$ by setting $(v_1,\ v_2,\ ...,\ v_n)<(v_1',\ v_2',\ ...,\ v_n')$ if either 
\begin{align*}
    \max_{1\le i\le n} \{v_i\}<\max_{1\le i\le n} \{v_i'\} 
\end{align*}
or
\begin{center}
    $\max_{1\le i\le n} \{v_i\}=\max_{1\le i\le n} \{v_i'\}$; $v_j<v_j'$ for some $1\le j\le n$; $v_\ell=v_\ell'$ for any $1\le \ell\le j-1$.
\end{center}


In \cite{VR}, Veomett and Radcliffe studied the strong product of double-sided infinite paths, and determined when the size of the vertex boundary of a $k$-element set in $V(\boxtimes_{i=1}^n P_\Z)$ is minimized. They also mentioned that a similar argument can be utilized to prove the following result about $\boxtimes_{i=1}^n P_{\N_0}$.


\begin{theorem}[Veomett and Radcliffe \cite{VR}]\label{thm3}
    Let $S$ be a set of $k$ vertices in $\boxtimes_{i=1}^n P_{\N_0}$. Then $|\partial_{\boxtimes_{i=1}^n P_{\N_0}}(S)|\ge |\partial_{\boxtimes_{i=1}^n P_{\N_0}}(S_0)|$, where $S_0$ is the set of the first $k$ vertices in the box order.
\end{theorem}

Whereas the result on $\square_{i=1}^n P_{\N_0}$ in Theorem \ref{thm1} can be directly extended to the result on $\square_{i=1}^n P_m$ in Theorem \ref{thm2}, the result on $\boxtimes_{i=1}^n P_{\N_0}$ in Theorem \ref{thm3} cannot be extended to $\boxtimes_{i=1}^n P_m$, or more generally, $\boxtimes_{i=1}^n P_{m_i}$. For example, suppose we have four vertices in $P_4\boxtimes P_4$. The set of the first four vertices in the box order is $\{(0,\ 0),\ (0,\ 1),\ (1,\ 0),\ (1,\ 1)\}$, and its vertex boundary is $\{(0,\ 2),\ (1,\ 2),\ (2,\ 2),\ (2,\ 1),\ (2,\ 0)\}$, which has five elements. But if we take the four-element set to be $\{(0,\ 0),\ (1,\ 0),\ (2,\ 0),\ (3,\ 0)\}$, then the vertex boundary is $\{(0,\ 1),\ (1,\ 1),\ (2,\ 1),\ (3,\ 1)\}$, which only has four elements.

Studying the general $n$-dimensional case about $\boxtimes_{i=1}^n P_{m_i}$ appears to be very challenging. In this paper, we handle the $2$-dimensional case. For $x,\ y,\ k\in \N$ with $2\le x\le y$ and $k\le xy$, we determine the cases in which a $k$-element set in $V(P_x\boxtimes P_y)$ has the minimum vertex boundary.

For $x,\ y,\ k\in \N$ with $2\le x\le y$ and $k\le xy$, let 
\begin{align*}
    I_1&=[1,\ (x-1)^2], \\
    I_2&=[x,\ x(y-1)],
\end{align*}
\[
    I_3=\begin{cases}
        [xy-(x-1)^2-(x-1),\ xy] & if\ y\ge x+1, \\
        [xy-(x-1)^2,\ xy] & if\ y=x,
    \end{cases}
\]
and let
    \[
        \alpha_1=\begin{cases}
            2\sqrt{k}+1 & if\ \sqrt{|S|}=\sqrt{k}\in \N, \\
            2\lfloor\sqrt{k}\rfloor+2 & if \lfloor\sqrt{k}\rfloor^2+1\le k\le \lfloor\sqrt{k}\rfloor^2+\lfloor\sqrt{k}\rfloor, \\
            2\lfloor\sqrt{k}\rfloor+3 & if \lfloor\sqrt{k}\rfloor^2+\lfloor\sqrt{k}\rfloor+1\le k\le \lfloor\sqrt{k}\rfloor^2+2\lfloor\sqrt{k}\rfloor,
        \end{cases}
    \]

    \[
        \alpha_2=\begin{cases}
            x & if\ x\mid k, \\
            x+1 & if\ x\nmid k,
        \end{cases}
    \]

    \[
        \alpha_3=\begin{cases}
            2\sqrt{K}-1 & if\ \sqrt{K}\in \N, \\
            2\lfloor\sqrt{K}\rfloor & if \lfloor\sqrt{K}\rfloor^2+1\le K\le \lfloor\sqrt{K}\rfloor^2+\lfloor\sqrt{K}\rfloor, \\
            2\lfloor\sqrt{K}\rfloor+1 & if \lfloor\sqrt{K}\rfloor^2+\lfloor\sqrt{K}\rfloor+1\le K\le \lfloor\sqrt{K}\rfloor^2+2\lfloor\sqrt{K}\rfloor,
        \end{cases}
    \]
    where $K:=xy-k$.

\begin{theorem}\label{finite}
We have
\begin{align*}
    \min_{|S|=k}|\partial_{P_x \boxtimes P_y}(S)|=\min_{i\in \mathcal{A}_k}\{\alpha_i\},
\end{align*}
where $\mathcal{A}_k$ is the subset of $\{1,\ 2,\ 3\}$ such that $k\in I_i$ for $i\in \mathcal{A}_k$ and $k\notin I_j$ for $j\notin \mathcal{A}_k$.
\end{theorem}

To help readers better understand, we present an illustrative example here.

\begin{example}
    Let $x=5$ and $y=6$. Then $I_1=[1,\ 16]$, $I_2=[5,\ 25]$, and $I_3=[10,\ 30]$.
    \begin{itemize}
        \item If we take $k=4$, then $\mathcal{A}_4=\{1\}$ and $\alpha_1=5$, so $\min_{|S|=4}|\partial_{P_5 \boxtimes P_6}(S)|=5$.
        \item If we take $k=13$, then $\mathcal{A}_{13}=\{1,\ 2,\ 3\}$, $\alpha_1=9$, $\alpha_2=6$, and $\alpha_3=8$, so $\min_{|S|=13}|\partial_{P_5 \boxtimes P_6}(S)|=6$.
        \item If we take $k=23$, then $\mathcal{A}_{23}=\{2,\ 3\}$, $\alpha_2=6$, and $\alpha_3=5$, so $\min_{|S|=23}|\partial_{P_5 \boxtimes P_6}(S)|=5$.
    \end{itemize}
\end{example}
    Figure \ref{pic} illustrates sets with minimum boundaries, where the vertices in the $k$-element subsets are black, the vertices in the vertex boundaries are yellow, and other vertices are white.
    
    \begin{figure}[H]
        \includegraphics[width=0.8\linewidth]{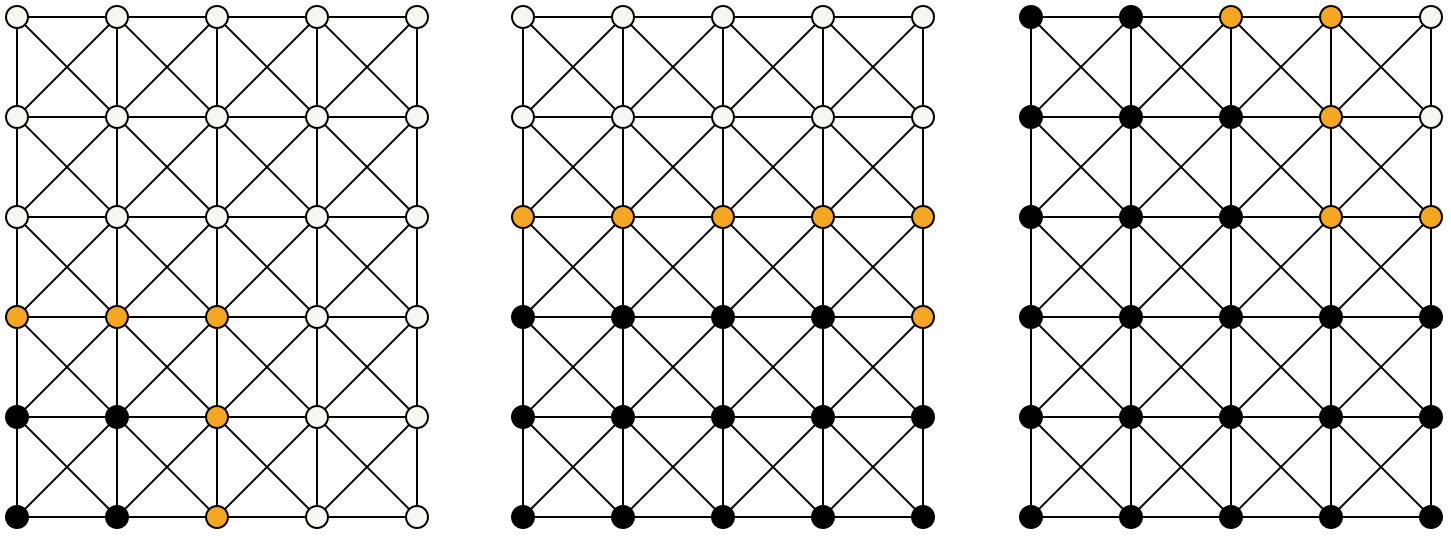}
        \caption{$k=4$, $k=13$, and $k=23$.}
        \label{pic}
    \end{figure}

Using the result in Theorem \ref{finite}, we can prove the following corollary about the strong product of a finite path and an infinite path.

\begin{corollary}\label{cor}
    We have
    \[
    \min_{|S|=k}|\partial_{P_x \boxtimes P_{\N_0}}(S)|=\begin{cases}
        \alpha_1 & if\ k\le x-1, \\
        \min\{\alpha_1,\ \alpha_2\} & if\ x\le k\le (x-1)^2, \\
        \alpha_2 & if\ k\ge (x-1)^2+1.
    \end{cases}
    \]
\end{corollary}

In the proof of Theorem \ref{finite}, we need to use the following two lemmas, where Lemma \ref{2d} is the $2$-dimensional case of Theorem \ref{thm3}.

\begin{lemma}\label{2d}
    Let $S$ be a set of $k$ vertices in $P_{\N_0} \boxtimes P_{\N_0}$. Then $|\partial_{P_{\N_0} \boxtimes P_{\N_0}}(S)|\ge |\partial_{P_{\N_0} \boxtimes P_{\N_0}}(S_0)|$, where $S_0$ is the set of the first $k$ vertices in the box order. Precisely, we have that 
    \[
        |\partial_{P_{\N_0} \boxtimes P_{\N_0}}(S)|\ge\begin{cases}
            2\sqrt{k}+1 & if\ \sqrt{|S|}=\sqrt{k}\in \N, \\
            2\lfloor\sqrt{k}\rfloor+2 & if \lfloor\sqrt{k}\rfloor^2+1\le k\le \lfloor\sqrt{k}\rfloor^2+\lfloor\sqrt{k}\rfloor, \\
            2\lfloor\sqrt{k}\rfloor+3 & if \lfloor\sqrt{k}\rfloor^2+\lfloor\sqrt{k}\rfloor+1\le k\le \lfloor\sqrt{k}\rfloor^2+2\lfloor\sqrt{k}\rfloor,
        \end{cases}
    \]
    where the equality can be attained when $S=S_0$.
\end{lemma}

For a graph $G=(V,\ E)$ and a set $S\subseteq V$, let the \emph{inner vertex boundary} of $S$, denoted by $\partial_G^{in}(S)$, be the set of vertices in $S$ that have at least one neighbor outside $S$. So we have
\begin{align*}
    \partial_G^{in}(S)=\{v\in S: N(v)\cap(V\setminus S)\neq \emptyset\}.
\end{align*}

\begin{lemma}\label{2din}
    Let $S$ be a set of $k$ vertices in $P_{\N_0} \boxtimes P_{\N_0}$. Then $|\partial_{P_{\N_0} \boxtimes P_{\N_0}}^{in}(S)|\ge |\partial_{P_{\N_0} \boxtimes P_{\N_0}}^{in}(S_0)|$, where $S_0$ is the set of the first $k$ vertices in the box order. Precisely, we have that 
    \[
        |\partial_{P_{\N_0} \boxtimes P_{\N_0}}^{in}(S)|\ge\begin{cases}
            2\sqrt{k}-1 & if\ \sqrt{|S|}=\sqrt{k}\in \N, \\
            2\lfloor\sqrt{k}\rfloor & if \lfloor\sqrt{k}\rfloor^2+1\le k\le \lfloor\sqrt{k}\rfloor^2+\lfloor\sqrt{k}\rfloor, \\
            2\lfloor\sqrt{k}\rfloor+1 & if \lfloor\sqrt{k}\rfloor^2+\lfloor\sqrt{k}\rfloor+1\le k\le \lfloor\sqrt{k}\rfloor^2+2\lfloor\sqrt{k}\rfloor,
        \end{cases}
    \]
    where the equality can be attained when $S=S_0$.
\end{lemma}

Although Lemma \ref{2d} is just a special case of Theorem \ref{thm3}, Theorem \ref{thm3} itself was only mentioned in \cite{VR} without a formal proof. So for completeness, we show proofs of both lemmas.

The remainder of this paper is organized as follows. In Section 2, we prove Lemma \ref{2d} and Lemma \ref{2din}. In Section 3, we first prove Theorem \ref{finite}, where the proof is split into five parts, and then prove Corollary \ref{cor}. In Section 4, we explain why the tensor products of paths are not studied in this paper, and suggest two problems.

\section{Proofs of Lemma \ref{2d} and Lemma \ref{2din}}
We can embed $P_{\N_0} \boxtimes P_{\N_0}$ and $P_x\boxtimes P_y$ onto the $x$-axis, the $y$-axis, and the first quadrant in a Cartesian coordinate system. From now on, in $P_{\N_0} \boxtimes P_{\N_0}$ or $P_x\boxtimes P_y$, the column with $(i,\ 0)$ on it will be called the $i$-th column, and the row with $(0,\ j)$ on it will be called the $j$-th row, so we will have the $0$-th column and the $0$-th row.

In a graph $G=(V,\ E)$, the \emph{closed neighborhood} of $U\subseteq V$, denoted by $N[U]$, is defined by
\begin{align*}
    N[U]:=U\cup \partial_G(U).
\end{align*}
So if we want to check the size of the vertex boundary of a $k$-element set $S$, it is equivalent to check the size of $N[S]$, which is $|S\cup \partial_G(S)|$, because $|\partial_G(S)|=|S\cup \partial_G(S)|-k$.

Let $U$ be a finite set of vertices in $P_{\N_0} \boxtimes P_{\N_0}$. Assume that $U$ has $y_i$ vertices on the $i$-th column for $0\le i\le I$, no vertices on the $i'$-th column for any $i'\ge I+1$, $x_j$ vertices on the $j$-th row for $0\le j\le J$, and no vertices on the $j'$-th row for any $j'\ge J+1$. We have $\sum_{i=0}^I y_i=\sum_{j=0}^J x_j=|U|$. Let the \emph{down-compression} $C_D(U)$ of $U$ be 
\begin{align*}
    \{(0,\ 0),\ (0,\ 1),\ ...,\ (0,\ y_0),\ (1,\ 0),\ (1,\ 1),\ ...,\ (1,\ y_1),\ ...,\ (I,\ 0),\ (I,\ 1),\ ...,\ (I,\ y_I)\}.
\end{align*}
Let the \emph{left-compression} $C_L(U)$ of $U$ be 
\begin{align*}
    \{(0,\ 0),\ (1,\ 0),\ ...,\ (x_0,\ 0),\ (0,\ 1),\ (1,\ 1),\ ...,\ (x_1,\ 1),\ ...,\ (0,\ J),\ (1,\ J),\ ...,\ (x_J,\ J)\}.
\end{align*}
Intuitively, the down-compression pushes all vertices downward, and the left-compression pushes all vertices to the left.

We show that the compressions do not increase the size of the closed neighborhood of a set. The proof of this lemma is similar to how Iamphongsai and Kittipassorn proved Lemma 4 in \cite{IK}.

\begin{lemma}\label{compression}
    Let $U$ be a finite set of vertices in $P_{\N_0} \boxtimes P_{\N_0}$. Then $|N[C_D(U)]|\le |N[U]|$ and $|N[C_L(U)]|\le |N[U]|$.
\end{lemma}

\begin{proof}
    We only prove $|N[C_D(U)]|\le |N[U]|$, because it is exactly the same for $|N[C_L(U)]|$. For a set $A\subseteq \N_0$, we define $A^+:=\{a+1:\ a\in A\}$, and define $A^-:=\{a-1:\ a\in A\}\cap \N_0$. For a finite set $B\subseteq V(P_{\N_0} \boxtimes P_{\N_0})$, we define $B_i:=\{\ell:\ (i,\ \ell)\in S\}$ for any $i\in \N_0$. Assume that $I$ is the largest number $i$ such that $U$ has a vertex on the $i$-th column. 
    
    By the adjacency relation in $P_{\N_0} \boxtimes P_{\N_0}$, we can see that, for any $0\le i\le I$,
    \begin{align*}
        N[U]_i=(U_i\cup U_i^+\cup U_i^-)\cup (U_{i-1}\cup U_{i-1}^+\cup U_{i-1}^-)\cup (U_{i+1}\cup U_{i+1}^+\cup U_{i+1}^-),
    \end{align*}
    where $U_{-1}$, $U_{-1}^+$, and $U_{-1}^-$ are taken to be $\emptyset$.

    Similarly, if we denote $C_D(U)$ by $W$, then
    \begin{align*}
        N[C_D(U)]_i=N[W]_i=(W_i\cup W_i^+\cup W_i^-)\cup (W_{i-1}\cup W_{i-1}^+\cup W_{i-1}^-)\cup (W_{i+1}\cup W_{i+1}^+\cup W_{i+1}^-),
    \end{align*}
    where $W_{-1}$, $W_{-1}^+$, and $W_{-1}^-$ are taken to be $\emptyset$.

    As $W$ is the down-compression of $U$, for each $0\le i\le I$, we have $W_i\cup W_i^+\cup W_i^-=W_i\cup W_i^+=\{0,\ 1,\ 2,\ ...,\ c_i\}$, where $c_i=|U_i|$. So we know 
    \begin{align*}
        |N[W]_i|&=\max\{|W_i\cup W_i^+\cup W_i^-|,\ |W_{i-1}\cup W_{i-1}^+\cup W_{i-1}^-|,\ |W_{i+1}\cup W_{i+1}^+\cup W_{i+1}^-|\} \\
        &=\max\{|U_i|+1,\ |U_{i-1}|+1,\ |U_{i+1}|+1\}.
    \end{align*}
    For each $0\le i\le I$, we have $|U_i\cup U_i^+\cup U_i^-|\ge |U_i\cup U_i^+|\ge |U_i|+1$. So $|N[W]_i|\le |N[U]_i|$, which further implies
    \begin{align*}
        |N[C_D(U)]|=|N[W]|=\sum_{i=0}^I |N[W]_i|\le \sum_{i=0}^I |N[U]_i|=|N[U]|.
    \end{align*}
\end{proof}

It is easy to see that, if $C_D(U)\neq U$, then
\begin{align*}
    \sum_{(x,\ y)\in C_D(U)}(x+y)<\sum_{(x,\ y)\in U}(x+y),
\end{align*}
and if $C_L(U)\neq U$, then
\begin{align*}
    \sum_{(x,\ y)\in C_L(U)}(x+y)<\sum_{(x,\ y)\in U}(x+y).
\end{align*}
So if we start from a finite set of vertices, and alternatively compress it down and compress it to the left, then this iteration will stop after finitely many steps, as the sum of $(x+y)$ over all $(x,\ y)$ in the set must be positive. We say a finite vertex set $U$ is \emph{compressed} if $C_D(U)=C_L(U)=U$. In fact, for any $t\in \N$, the set of the first $t$ vertices in the box order is compressed.

\begin{proof}[Proof of Lemma \ref{2d}]
    Let $\mathcal{C}$ be the set of all compressed $k$-element sets. Lemma \ref{compression} shows that the compressions do not increase the size of the closed neighborhood of a set. So, to prove Lemma \ref{2d}, we only need to show that $\min_{S\in \mathcal{C}}|N[S]|=|N[S_0]|$, that is, $\min_{S\in \mathcal{C}}|N[S]|$ can be attained by the set of the first $k$ vertices in the box order.
    
    There is a $t\in \N$ such that $t^2\le k\le (t+1)^2-1$. We first show that $\min_{S\in \mathcal{C}}|N[S]|$ can be attained by a set containing every vertex in $\{(v_1,\ v_2):\ 0\le v_1,\ v_2\le t-1\}$.
    
    Let $S$ be a set in $\mathcal{C}$. If $\{(v_1,\ v_2):\ 0\le v_1,\ v_2\le t-1\}\subseteq S$, then we can skip this part, so let us assume $\{(v_1,\ v_2):\ 0\le v_1,\ v_2\le t-1\}\setminus S\neq \emptyset$. Because $|S|=k\ge t^2$ and it is a compressed set, we must have that $\{(x,\ 0):\ 0\le x\le t-1\}\subseteq S$ or $\{(0,\ y):\ 0\le y\le t-1\}\subseteq S$. By symmetry, we assume $\{(0,\ y):\ 0\le y\le t-1\}\subseteq S$. Then there is a $1\le j\le t-1$ such that $\{(s,\ y):\ 0\le y\le t-1\}\nsubseteq S$ for any $j\le s\le t-1$ but $\{(\ell,\ y):\ 0\le y\le t-1\}\subseteq S$ for any $0\le \ell \le j-1$. We have two cases.

    \textbf{Case 1.} $(j,\ 0)\in S$.

    We assume $(j,\ 0),\ (j,\ 1),\ (j,\ 2),\ ...,\ (j,\ r-1)\in S$, but $(j,\ r)\notin S$. As $|S|\ge t^2$ and $S$ is compressed, there is a vertex $(u_1,\ u_2)\in S$ outside $\{(v_1,\ v_2):\ 0\le v_1,\ v_2\le t-1\}$ with $(u_1+1,\ u_2)\notin S$ and $(u_1,\ u_2+1)\notin S$. Now we exclude this vertex $(u_1,\ u_2)$ from $S$, add $(j,\ r)$ to $S$, and get a new set $S_{new}$. Excluding $(u_1,\ u_2)$ decreases $|N[S]|$ by one or two. In fact, it is one if $(u_1-1,\ u_2)\in S$ and $(u_1,\ u_2-1)\in S$, and it is two if either $(u_1-1,\ u_2)\notin S$ or $(u_1,\ u_2-1)\notin S$, which means $u_1=0$ or $u_2=0$. Adding $(j,\ r)$ increases $|N[S]|$ by exactly one. So we have $|N[S_{new}]|\le |N[S]|$. Also, it is easy to see that $S_{new}$ is still a compressed set. So we can update $S$ to $S_{new}$ and repeat this process until $\{(j,\ y):\ 0\le y\le t-1\}\subseteq S$, and then we can update $j$ to $j+1$ and again we have the same two cases.

    \textbf{Case 2.} $(j,\ 0)\notin S$.

    In this case, our current $S\cap \{(v_1,\ v_2):\ 0\le v_1,\ v_2\le t-1\}$ looks like a rectangle with base $j$ and height $t$. We will rotate and reflect the shape of $S\setminus \{(v_1,\ v_2):\ 0\le v_1,\ v_2\le t-1\}$, which is now on the top of the rectangle, and move it to the right of the rectangle.

    For any vertex $(u_1,\ u_2)\in S\setminus \{(v_1,\ v_2):\ 0\le v_1,\ v_2\le t-1\}$, we exclude this vertex from $S$, and add $(u_2-(t-1)+(j-1),\ u_1)=(u_2-t+j,\ u_1)$ to $S$. After moving all vertices in $S\setminus \{(v_1,\ v_2):\ 0\le v_1,\ v_2\le t-1\}$, we get a new set $S_{new}$. It is easy to see that $|N[S]|=|N[S_{new}]|$. Here we make two observations. First, $S\setminus \{(v_1,\ v_2):\ 0\le v_1,\ v_2\le t-1\}$ is not very wide, because $\max \{u_1:\ (u_1,\ u_2)\in S\setminus \{(v_1,\ v_2):\ 0\le v_1,\ v_2\le t-1\}\}=\max \{u_1:\ (u_1,\ t)\in S\}$ is at most $j-1<t-1$. This implies that our $S_{new}$ is still a compressed set. Second, $S\setminus \{(v_1,\ v_2):\ 0\le v_1,\ v_2\le t-1\}$ is pretty tall, because 
    \begin{align*}
        &[\max \{u_2:\ (0,\ u_2)\in S\}-(t-1)]\times j+t\times j \\
        \ge &|S\setminus \{(v_1,\ v_2):\ 0\le v_1,\ v_2\le t-1\}|+|S\cap \{(v_1,\ v_2):\ 0\le v_1,\ v_2\le t-1\}| \\
        = &|S|\\
        \ge &t^2,
    \end{align*}
    which means $\max \{u_2:\ (0,\ u_2)\in S\}\ge \frac{t^2}{j}-1$. From $S$ to $S_{new}$, the vertex $\Bigl(0,\ \Bigl\lfloor \frac{t^2}{j}-1 \Bigr\rfloor\Bigr)$ is excluded, the vertex $\Bigl(\Bigl\lfloor \frac{t^2}{j}-1 \Bigr\rfloor-t+j,\ 0 \Bigr)$ is added, and we have that $\Bigl\lfloor \frac{t^2}{j}-1 \Bigr\rfloor-t+j=\Bigl\lfloor \frac{t^2}{j}+j-t-1 \Bigr\rfloor \ge \lfloor t-1\rfloor=t-1$, so actually for any $0\le x\le t-1$, we have $(x,\ 0)\in S_{new}$.

    Now we update $S$ to $S_{new}$, and we have $(j,\ 0)\in S$, so we will go back to Case 1. In fact, because now $(x,\ 0)\in S$ for any $0\le x\le t-1$, Case 2 will never be encountered again.

    In each case, we exclude at least one vertex outside $\{(v_1,\ v_2):\ 0\le v_1,\ v_2\le t-1\}$ from $S$ and add at least one vertex in $\{(v_1,\ v_2):\ 0\le v_1,\ v_2\le t-1\}$ into $S$ while making sure that the new set $S_{new}$ we get is still a compressed set and $|N[S_{new}]|\le |N[S]|$, which means $|\partial_{P_{\N_0} \boxtimes P_{\N_0}}(S_{new})|\le |\partial_{P_{\N_0} \boxtimes P_{\N_0}}(S)|$. This iteration will stop after finitely many steps, and at that point, the new set we get will contain every vertex in $\{(v_1,\ v_2):\ 0\le v_1,\ v_2\le t-1\}$, and the size of the vertex boundary of the new set will be at most the size of the vertex boundary of the original set.

    Now we have shown that $\min_{S\in \mathcal{C}}|N[S]|$ can be attained by a set containing every vertex in $\{(v_1,\ v_2):\ 0\le v_1,\ v_2\le t-1\}$. If $k=t^2$, then the vertices in $\{(v_1,\ v_2):\ 0\le v_1,\ v_2\le t-1\}$ are exactly the first $t^2$ vertices in the box order, and we have the desired conclusion. Assume $t^2+1\le k\le (t+1)^2-1$. Let $\mathcal{C}'\subseteq \mathcal{C}$ be the set of those compressed $k$-element sets containing every vertex in $\{(v_1,\ v_2):\ 0\le v_1,\ v_2\le t-1\}$. Let $S$ be a set in $\mathcal{C}'$.
    
    As $S$ is compressed, we have $(t,\ 0)\in S$ or $(0,\ t)\in S$. Here we only show how to address the case where $(0,\ t)\in S$. Because if $(t,\ 0)\in S$ and $(0,\ t)\notin S$, then we can do the same argument symmetrically, and at the end, we exclude every $(x,\ y)\in S$ from $S$ and add $(y,\ x)$ to $S$, and we will get the same conclusion.
    
    Now we have two cases. This part of argument is similar to what we have already done.

    \textbf{Case I.} There is some $1\le j\le t-1$ such that $(s,\ t)\notin S$ for any $j\le s\le t-1$ and $(\ell,\ t)\in S$ for any $0\le \ell\le j-1$.

    If now we have $S=\{(v_1,\ v_2):\ 0\le v_1,\ v_2\le t-1\}\cup \{(\ell,\ t):\ 0\le \ell\le j-1\}$, then the vertices in $S$ are the first $k=t^2+j$ vertices in the box order. Otherwise, there is a vertex $(u_1,\ u_2)\in S$ outside $\{(v_1,\ v_2):\ 0\le v_1,\ v_2\le t-1\}\cup \{(\ell,\ t):\ 0\le \ell\le j-1\}$ with $(u_1+1,\ u_2)\notin S$ and $(u_1,\ u_2+1)\notin S$. Now we exclude this vertex $(u_1,\ u_2)$ from $S$, add $(j,\ t)$ to $S$, and get a new set $S_{new}$. Excluding $(u_1,\ u_2)$ decreases $|N[S]|$ by one or two. Adding $(j,\ t)$ increases $|N[S]|$ by exactly one. So we have $|N[S_{new}]|\le |N[S]|$. Also, it is easy to see that $S_{new}$ is still a compressed set. So we can update $S$ to $S_{new}$, update $j$ to $j+1$, and repeat this process until $S=\{(v_1,\ v_2):\ 0\le v_1,\ v_2\le t-1\}\cup \{(\ell,\ t):\ 0\le \ell\le j-1\}$ or $\{(\ell,\ t):\ 0\le \ell\le t-1\}\subseteq S$. In the former case, we get the desired conclusion; in the later case, we get into Case II.

    \textbf{Case II.} $(\ell,\ t)\in S$ for any $0\le \ell \le t-1$.

    If now we have $S=\{(v_1,\ v_2):\ 0\le v_1,\ v_2\le t-1\}\cup \{(\ell,\ t):\ 0\le \ell\le t-1\}$, then the vertices in $S$ are the first $k=t^2+t$ vertices in the box order. Otherwise, we have two subcases.

    \textbf{Subcase II.i.} $(t,\ 0)\in S$.

    In this subcase, we just need to mimic what we have done in Case I to exclude vertices outside $\{(w_1,\ w_2):\ 0\le w_1,\ w_2\le t\}$ from $S$ and add vertices on the $t$-th column to $S$. This process will stop after at most $t-1$ steps, because $|S|=k\le (t+1)^2-1$. When we stop, $S$ will have the first $k$ elements in the box order.

    \textbf{Subcase II.ii.} $(t,\ 0)\notin S$.

    In this subcase, $(0,\ t+1)\in S$, as $S$ is a compressed set. Just like what we have done in Case 2, for any vertex $(u_1,\ u_2)\in S\setminus \{(v_1,\ v_2):\ 0\le v_1\le t-1,\ 0\le v_2\le t\}$, we exclude this vertex from $S$, and add $(u_2-t+(t-1),\ u_1)=(u_2-1,\ u_1)$ to $S$. In this process, $|N[S]|$ stays the same, and the new set we get is still a compressed set. Moreover, $(0,\ t+1)\in S$ is excluded and $(t,\ 0)$ is added, so we can go to Subcase II.i, and eventually we will get the first $k$ elements in the box order.

    As a summary, in this proof, we began with a set $S\in \mathcal{C}$, and showed that we could do some operations to make $S$ become $S_0$, the set of the first $k$ elements in the box order, and each of these operations did not increase $|N[S]|$. This argument works for any $S\in \mathcal{C}$. Thus, we have $\min_{S\in \mathcal{C}}|N[S]|=|N[S_0]|$. It is easy to check that 
    \[
        |N[S_0]|=\begin{cases}
        (t+1)^2 & if\ |S_0|=k=t^2, \\
        (t+1)^2+k-t^2+1 & if\ t^2+1\le |S_0|=k\le t^2+t, \\
        (t+1)^2+k-t^2+2 & if\ t^2+t+1\le |S_0|=k\le t^2+2t.
        \end{cases}
    \]
    Combining this result with the fact that $|\partial_{P_{\N_0} \boxtimes P_{\N_0}}(S_0)|=|N[S_0]|-k$, we have
    \[
        |\partial_{P_{\N_0} \boxtimes P_{\N_0}}(S)|\ge\begin{cases}
            2\sqrt{k}+1 & if\ \sqrt{|S|}=\sqrt{k}\in \N, \\
            2\lfloor\sqrt{k}\rfloor+2 & if \lfloor\sqrt{k}\rfloor^2+1\le k\le \lfloor\sqrt{k}\rfloor^2+\lfloor\sqrt{k}\rfloor, \\
            2\lfloor\sqrt{k}\rfloor+3 & if \lfloor\sqrt{k}\rfloor^2+\lfloor\sqrt{k}\rfloor+1\le k\le \lfloor\sqrt{k}\rfloor^2+2\lfloor\sqrt{k}\rfloor.
        \end{cases}
    \]
\end{proof}

Then, to prove Lemma \ref{2din}, first we show that the compressions do not increase the size of the inner vertex boundary of a set. We define the \emph{interior} of $S$, denoted by $Int(S)$, to be $S\setminus \partial_G^{in}(S)$. It is equivalent to show that the compressions do not decrease the size of the interior of a set. The method for proving Lemma \ref{incompression} is not the same as what we had in the proof of Lemma \ref{compression}.

\begin{lemma}\label{incompression}
    Let $U$ be a finite set of vertices in $P_{\N_0} \boxtimes P_{\N_0}$. Then $|Int(C_D(U))|\ge |Int(U)|$ and $|Int(C_L(U))|\ge |Int(U)|$.
\end{lemma}

\begin{proof}
    We only prove $|Int(C_D(U))|\ge |Int(U)|$, because it is the same for $|Int(C_L(U))|$. Assume that $U$ has $y_i$ vertices on the $i$-th column for any $i\in \N_0$, and assume that $I$ is the largest $i$ such that $U$ has a vertex on the $i$-th column. So for any $i\ge I+1$, we have $y_i=0$. Additionally, we define $y_{-1}=\max_{0\le i\le I} y_i+1$ (any larger number also works).

    As $C_D(U)$ is the down-compression of $U$, we can see that, for any $0\le i\le I$, the number of vertices in $Int(C_D(U))$ on the $i$-th column is 
    \begin{align*}
        \max \{0,\ \min \{y_{i-1}-1,\ y_i-1,\ y_{i+1}-1\}\}.
    \end{align*}

    Now we look at $Int(U)$. Assume that for some $0\le i\le I$, $Int(U)$ has $k_i$ vertices on the $i$-th column with $k_i>0$ and $k_i>\min \{y_{i-1}-1,\ y_i-1,\ y_{i+1}-1\}$. Assume that these $k_i$ vertices are $(i,\ s_1),\ (i,\ s_2),\ ...,\ (i,\ s_{k_i})$ with $s_1<s_2<...<s_{k_i}$. These vertices are in the interior of $U$, so by the adjacency relation in $P_{\N_0} \boxtimes P_{\N_0}$, on the $i$-th column, we must have $(i,\ s_1),\ (i,\ s_2),\ ...,\ (i,\ s_{k_i})\in U$ and $(i,\ s_{k_i}+1)\in U$; on the $(i-1)$-th column, we must have $(i-1,\ s_1),\ (i-1,\ s_2),\ ...,\ (i-1,\ s_{k_i})\in U$ and $(i-1,\ s_{k_i}+1)\in U$; on the $(i+1)$-th column, we must have $(i+1,\ s_1),\ (i+1,\ s_2),\ ...,\ (i+1,\ s_{k_i})\in U$ and $(i+1,\ s_{k_i}+1)\in U$. But then, on each of the $(i-1)$-th, $i$-th, and $(i+1)$-th columns, we have at least $k_i+1>\min\{y_{i-1},\ y_i,\ y_{i+1}\}$ vertices, which is not possible.

    So, for any $0\le i\le I$, $Int(U)$ has at most $\max \{0,\ \min \{y_{i-1}-1,\ y_i-1,\ y_{i+1}-1\}\}$ vertices on the $i$-th colomn. Thus, $|Int(C_D(U))|\ge |Int(U)|$.
\end{proof}

Then, to complete the proof of Lemma \ref{2din}, we just need to mimic the proof of Lemma \ref{2d}. This part is omitted here.

\section{Proof of Theorem \ref{finite}}

Now we take $U$ to be a set of vertices in $P_x\boxtimes P_y$, where $x,\ y\in \N$ and $2\le x\le y$. We can define the down-compression $C_D(U)$ of $U$ and the left-compression $C_L(U)$ of $U$ in the same way we defined them in $P_{\N_0}\boxtimes P_{\N_0}$. Also, we say $U$ is compressed if $C_D(U)=C_L(U)=U$.

Very similar to how we proved Lemma \ref{compression}, we can prove the following conclusion.

\begin{lemma}
    Let $U$ be a set of vertices in $P_x \boxtimes P_y$. Then $|N[C_D(U)]|\le |N[U]|$ and $|N[C_L(U)]|\le |N[U]|$.
\end{lemma}

So, to determine when the size of the vertex boundary of a $k$-element set in $V(P_x \boxtimes P_y)$ is minimized, it suffices to consider those compressed sets.

Let $S$ be a compressed $k$-element set in $V(P_x \boxtimes P_y)$. We split the proof of Theorem \ref{finite} into five parts: In each of the first four parts, we address a specific case. In the fifth part, we determine when $\min|\partial_{P_x \boxtimes P_y}(S)|$ is attained.

\begin{proof}[Proof of Theorem \ref{finite}]
~\\
\textbf{Part 1.} We address the first case: Neither $(x-1,\ 0)$ nor $(0,\ y-1)$ is in $S$. This case is only possible when $k\le (x-1)(y-1)$.

By Lemma \ref{2d}, we can see that if $(x-1,\ 0)\notin S$ and $(0,\ y-1)\notin S$, then $\partial_{P_x \boxtimes P_y}(S)\subseteq P_x \boxtimes P_y$ and $|\partial_{P_x \boxtimes P_y}(S)|=|\partial_{P_{\N_0} \boxtimes P_{\N_0}}(S)|\ge |\partial_{P_{\N_0} \boxtimes P_{\N_0}}(S_0)|$, where $S_0$ is the set of the first $k$ vertices in the box order. Precisely, we have that 
    \[
        |\partial_{P_x \boxtimes P_y}(S)|\ge\begin{cases}
            2\sqrt{k}+1 & if\ \sqrt{|S|}=\sqrt{k}\in \N, \\
            2\lfloor\sqrt{k}\rfloor+2 & if \lfloor\sqrt{k}\rfloor^2+1\le k\le \lfloor\sqrt{k}\rfloor^2+\lfloor\sqrt{k}\rfloor, \\
            2\lfloor\sqrt{k}\rfloor+3 & if \lfloor\sqrt{k}\rfloor^2+\lfloor\sqrt{k}\rfloor+1\le k\le \lfloor\sqrt{k}\rfloor^2+2\lfloor\sqrt{k}\rfloor.
        \end{cases}
    \]
    
Futhermore, if $k\le (x-1)^2$, then the equality can be attained, because $S_0$ is a subset of $\{(i,\ j):\ 0\le i,\ j\le x-2\}$, and the vertex boundary of $S_0$ is a subset of $\{(i,\ j):\ 0\le i,\ j\le x-1\}\subseteq V(P_x \boxtimes P_y)$. The case $k=4$ in Figure \ref{pic} is an example.

If $(x-1)^2+1\le k\le (x-1)(y-1)$, then it is possible that the equality cannot be attained. But in this case we have $|\partial_{P_x \boxtimes P_y}(S)|\ge 2x$, and we will show that for $k$ in this range, there is a better lower bound that can be attained.

\textbf{Part 2.} We address the second case: $(x-1,\ 0)\in S$ but $(0,\ y-1)\notin S$. This case is only possible when $x\le k\le x(y-1)$.

The \emph{co-lexicographic order} on the vertices in $P_x \boxtimes P_y$ is defined by setting $(v_1,\ v_2)<(v_1',\ v_2')$ if either
\begin{center}
    $v_2<v_2'$
\end{center}
or
\begin{center}
    $v_2=v_2'$ and $v_1<v_1'$.
\end{center}

We prove: If $(x-1,\ 0)\in S$ but $(0,\ y-1)\notin S$, then $|\partial_{P_x \boxtimes P_y}(S)|\ge |\partial_{P_x \boxtimes P_y}(S_1)|$, where $S_1$ is the set of the first $k$ vertices in the co-lexicographic order. Precisely, we have that 
    \[
        |\partial_{P_x \boxtimes P_y}(S)|\ge\begin{cases}
            x & if\ x\mid k, \\
            x+1 & if\ x\nmid k,
        \end{cases}
    \]
    where the equality can be attained when $S=S_1$. The case $k=13$ in Figure \ref{pic} is an example.

\begin{proof}
    Assume $S\neq S_1$. Then there is a vertex $(v_1,\ v_2)$ such that every vertex before $(v_1,\ v_2)$ in the co-lexicographic order is in $S$, but $(v_1,\ v_2)$ itself is not in $S$. As $S$ is compressed, any vertex $(x,\ y)$ with $x\ge v_1$ and $y\ge v_2$ cannot be in $S$. Note that $v_1\neq 0$, because otherwise $S$ consists of all the vertices below the $v_2$-th row and no vertices on the $v_2$-th row or above the $v_2$-th row, so $S$ has the first $|S|$ vertices in the co-lexicographic order, a contradiction. We also have $v_2\neq 0$, because every vertex on the $0$-th row must be in $S$.

    Now, as $S\neq S_1$, there is a set of vertices $A\subseteq S$ with every vertex in $A$ appearing after $(v_1,\ v_2)$ in the co-lexicographic order.

    \textbf{Claim.} For any $(a_1,\ a_2)\in A$, we have that $a_1\neq x-1$ and $a_2\neq y-1$.

    \begin{proof}[Proof of Claim.]
        If $a_2=y-1$, then as $S$ is compressed, we have $(0,\ y-1)\in S$, a contradiction.

        Assume $a_1=x-1$. If $a_2<v_2$, then $(a_1,\ a_2)$ appears before $(v_1,\ v_2)$ in the co-lexicographic order, a contradiction. If $a_2\ge v_2$, then we have $a_2\ge v_2$ and $a_1=x-1\ge v_1$, which contradicts the assumption that $S$ is compressed.
    \end{proof}

    Let $(b_1,\ b_2)$ be a vertex in $A$ such that $b_1+b_2=\max_{(a_1,\ a_2)\in A}\{a_1+a_2\}$. Now we exclude $(b_1,\ b_2)$ from $S$, add $(v_1,\ v_2)$ to $S$, and get a new set $S_{new}$. Excluding $(b_1,\ b_2)$ decreases $|N[S]|$ by one or two (two if $b_1=0$; one otherwise). Adding $(v_1,\ v_2)$ increases $|N[S]|$ by zero or one (zero if $v_1=x-1$; one otherwise), because $v_1\neq 0$ and $v_2\neq 0$. So in this process, $|N[S]|$ does not increase, which means $|\partial_{P_x \boxtimes P_y}(S)|$ does not increase. We can update $S$ to $S_{new}$ and repeat this process until $S=S_1$. So $|\partial_{P_x \boxtimes P_y}(S)|\ge |\partial_{P_x \boxtimes P_y}(S_1)|$.

    It is easy to check that
    \[
        |\partial_{P_x \boxtimes P_y}(S_1)|=\begin{cases}
            x & if\ x\mid k, \\
            x+1 & if\ x\nmid k.
        \end{cases}
    \]
\end{proof}

\textbf{Part 3.} The third case is that $(0,\ y-1)\in S$ but $(x-1,\ 0)\notin S$, which is only possible when $y\le k\le y(x-1)$. Similarly, we can prove that in this case we have
    \[
        |\partial_{P_x \boxtimes P_y}(S)|\ge\begin{cases}
            y & if\ y\mid k, \\
            y+1 & if\ y\nmid k.
        \end{cases}
    \]

\textbf{Part 4.} Let us address the fourth and final possible case: Both $(x-1,\ 0)$ and $(0,\ y-1)$ are in $S$. This case is only possible when $k\ge x+y-1$.

Let $K=xy-k$. We prove: If both $(x-1,\ 0)$ and $(0,\ y-1)$ are in $S$, then
    \[
        |\partial_{P_x \boxtimes P_y}(S)|\ge\begin{cases}
            2\sqrt{K}-1 & if\ \sqrt{K}\in \N, \\
            2\lfloor\sqrt{K}\rfloor & if \lfloor\sqrt{K}\rfloor^2+1\le K\le \lfloor\sqrt{K}\rfloor^2+\lfloor\sqrt{K}\rfloor, \\
            2\lfloor\sqrt{K}\rfloor+1 & if \lfloor\sqrt{K}\rfloor^2+\lfloor\sqrt{K}\rfloor+1\le K\le \lfloor\sqrt{K}\rfloor^2+2\lfloor\sqrt{K}\rfloor.
        \end{cases}
    \]

\begin{proof}
    We observe that
    \begin{align*}
        \partial_{P_x \boxtimes P_y}(S)=\partial_{P_x \boxtimes P_y}^{in}(\overline{S}),
    \end{align*}
    where $\overline{S}:=V(P_x \boxtimes P_y)\setminus S$.
    
    We have assumed that $S$ is a compressed set in $P_x \boxtimes P_y$, so $\overline{S}$ is also a "compressed" set: It is not down-left-compressed, but up-right-compressed. Similar to what we did in Part 1, we can directly apply Lemma \ref{2din} here to get the desired lower bound.
\end{proof}

Regarding whether the equality can be attained, we have two cases.

\textbf{Case 1.} $xy-k\le (x-1)^2$, which means $k\ge xy-(x-1)^2$. In this case, the equality can be attained. Indeed, it is attained when we let $S_0'$ be the set of the first $xy-k$ elements in the box order, where we have $S_0'\subseteq \{(i,\ j):\ 0\le i,\ j\le x-2\}$, and then take $\overline{S}=\{(x-i,\ y-j):\ (i,\ j)\in S_0'\}$. The case $k=23$ in Figure \ref{pic} is an example.

\textbf{Case 2.} $xy-k\ge (x-1)^2+1$, which means $x+y-1\le k\le xy-(x-1)^2-1$. In this case, we must have $y\ge x+1$, because if $y=x$, then $x+y-1=2x-1>x^2-(x-1)^2-1=xy-(x-1)^2-1$, a contradiction.

There are two subcases.

\textbf{Subcase 2.1.} $(x-1)^2+1\le xy-k\le (x-1)^2+(x-1)$, which means $xy-(x-1)^2-(x-1)\le k\le xy-(x-1)^2-1$. In this subcase, the equality can be attained. It is attained when we let $S_0'$ be the set of the first $xy-k$ elements in the box order, where we have $S_0'\subseteq \{(i,\ j):\ 0\le i\le x-2,\ 0\le j\le x-1\}$, and then take $\overline{S}=\{(x-i,\ y-j):\ (i,\ j)\in S_0'\}$. Because $y\ge x+1$, we know that the vertices on the $0$-th row or the $0$-th column are in $S=V(P_x \boxtimes P_y)\setminus \overline{S}$, and in particular, $(x-1,\ 0),\ (0,\ y-1)\in S$.

\textbf{Subcase 2.2.} $xy-k\ge (x-1)^2+(x-1)+1$, which means $x+y-1\le k\le xy-(x-1)^2-(x-1)-1$. In this subcase, it is possible that the equality cannot be attained. However, in this subcase we have $|\partial_{P_x \boxtimes P_y}(S)|\ge 2x-1$, because if $\lfloor\sqrt{K}\rfloor=\lfloor\sqrt{xy-k}\rfloor=x-1$, then $\lfloor\sqrt{K}\rfloor^2+\lfloor\sqrt{K}\rfloor+1=(x-1)^2+(x-1)+1\le xy-k=K$, so $|\partial_{P_x \boxtimes P_y}(S)|\ge 2(x-1)+1=2x-1$; if $\lfloor\sqrt{K}\rfloor\ge x$, then we always have $|\partial_{P_x \boxtimes P_y}(S)|\ge 2x-1$. We will show that for $k$ in this range, there is a lower bound better than $2x-1$ that can be attained.

\textbf{Part 5.} Given $x,\ y,\ k\in \N$ with $2\le x\le y$ and $1\le k\le xy$, what is $\min_{|S|=k}|\partial_{P_x \boxtimes P_y}(S)|$?

First, we make the following two observations, which tell us that the conclusion in Part 3 can be ignored.
\begin{itemize}
    \item The bound in Part 2 is always better than the bound in Part 3.
    \item in Part 2, we need $x\le k\le x(y-1)$; in Part 3, we need $y\le k\le y(x-1)$; we always have $[y,\ y(x-1)]\subseteq [x,\ x(y-1)]$.
\end{itemize}

Then, in Part 1, we have that if $(x-1)^2+1\le k\le (x-1)(y-1)$, then $|\partial_{P_x \boxtimes P_y}(S)|\ge 2x$, but the lower bound in Part 2, which can be attained, is always better than $2x$. This is because $x\le 2x$ and $x+1\le 2x$ for any $x\in \N$. We also have that $[(x-1)^2+1,\ (x-1)(y-1)]\subseteq [x,\ x(y-1)]$. So the only information we need to take from Part 1 is the lower bound for $1\le k\le (x-1)^2$.

In Part 4, we have that if $y\ge x+1$ and $x+y-1\le k\le xy-(x-1)^2-(x-1)-1$, then $|\partial_{P_x \boxtimes P_y}(S)|\ge 2x-1$, but the lower bound in Part 2, which can be attained, is always better than $2x-1$. This is because $x\le 2x-1$ and $x+1\le 2x-1$ for any $x\ge 2$. We also have that $[x+y-1,\ xy-(x-1)^2-(x-1)-1]\subseteq [x,\ x(y-1)]$. So the only information we need to take from Part 4 is the lower bound for $xy-(x-1)^2-(x-1)\le k\le xy$, if $y\ge x+1$; or the lower bound for $xy-(x-1)^2\le k\le xy$, if $y=x$.

Finally, let
\begin{align*}
    I_1&=[1,\ (x-1)^2], \\
    I_2&=[x,\ x(y-1)],
\end{align*}
\[
    I_3=\begin{cases}
        [xy-(x-1)^2-(x-1),\ xy] & if\ y\ge x+1, \\
        [xy-(x-1)^2,\ xy] & if\ y=x,
    \end{cases}
\]
and let the lower bound in Part 1 be $\alpha_1$, the lower bound in Part 2 be $\alpha_2$, and the lower bound in Part 4 be $\alpha_3$. Each $\alpha_t$ can be attained in $I_t$. Then we have
\begin{align*}
    \min_{|S|=k}|\partial_{P_x \boxtimes P_y}(S)|=\min_{i\in \mathcal{A}_k}\{\alpha_i\},
\end{align*}
where $\mathcal{A}_k$ is the subset of $\{1,\ 2,\ 3\}$ such that $k\in I_i$ for $i\in \mathcal{A}_k$ and $k\notin I_j$ for $j\notin \mathcal{A}_k$.
\end{proof}

Explicitly presenting this result in the form of a piecewise function is arduous and not essential to the main argument, but in general, if we focus on the big picture, then we can compare the values of $2\sqrt{k}$, $x$, and $2\sqrt{xy-k}$. It is clear that
    \[
        \min\{2\sqrt{k},\ x,\ 2\sqrt{xy-k}\}=\begin{cases}
            2\sqrt{k} & if\ 1\le k\le \frac{x^2}{4}, \\
            x & if\ \frac{x^2}{4}<k\le xy-\frac{x^2}{4}, \\
            2\sqrt{xy-k} & if\ xy-\frac{x^2}{4}<k\le xy.
        \end{cases}
    \]

Thus, if $x$ and $y$ are not very small, then
    \[
        \min_{|S|=k}|\partial_{P_x \boxtimes P_y}(S)|\approx\begin{cases}
            2\sqrt{k} & if\ 1\le k\le \frac{x^2}{4}, \\
            x & if\ \frac{x^2}{4}<k\le xy-\frac{x^2}{4}, \\
            2\sqrt{xy-k} & if\ xy-\frac{x^2}{4}<k\le xy.
        \end{cases}
    \]

To prove Corollary \ref{cor}, similarly, it suffices to consider those compressed $k$-element sets. 

\begin{proof}[Proof of Corollary \ref{cor}]
    Let $S$ be a compressed $k$-element set in $V(P_x \boxtimes P_{\N_0})$. Choose $r\in \N$ to be large enough such that $r\ge k+1$ and $xr-(x-1)^2-(x-1)\ge k+1$. We have $N[S]\subseteq \{(i,\ j): 0\le i\le x-1,\ 0\le j\le r-1\}$. Let $H\simeq P_x \boxtimes P_r$ be the subgraph of $P_x \boxtimes P_{\N_0}$ induced by $\{(i,\ j): 0\le i\le x-1,\ 0\le j\le r-1\}$. Then $\partial_{P_x \boxtimes P_{\N_0}}(S)=\partial_{H}(S)$. The minimum possible size of $\partial_{H}(S)$ is determined by Theorem \ref{finite}. Moreover, as $xr-(x-1)^2-(x-1)\ge k+1$, we know $k\notin I_3$. So, by comparing $I_1$ and $I_2$, we have
    \[
    \min_{|S|=k}|\partial_{P_x \boxtimes P_{\N_0}}(S)|=\min_{|S|=k}|\partial_{H}(S)|=\begin{cases}
        \alpha_1 & if\ k\le \min\{(x-1)^2,\ x-1\}, \\
        \min\{\alpha_1,\ \alpha_2\} & if\ x\le k\le (x-1)^2, \\
        \alpha_2 & if\ k\ge \max\{(x-1)^2+1,\ x\}.
    \end{cases}
    \]
    But for any $x\ge 2$, we have $\min\{(x-1)^2,\ x-1\}=x-1$ and $\max\{(x-1)^2+1,\ x\}=(x-1)^2+1$. Thus
    \[
    \min_{|S|=k}|\partial_{P_x \boxtimes P_{\N_0}}(S)|=\begin{cases}
        \alpha_1 & if\ k\le x-1, \\
        \min\{\alpha_1,\ \alpha_2\} & if\ x\le k\le (x-1)^2, \\
        \alpha_2 & if\ k\ge (x-1)^2+1.
    \end{cases}
    \]
\end{proof}

\section{Remarks}
One may wonder why we do not study the isoperimetric inequalities on the \emph{tensor products} of paths, where the tensor product of two graphs $G$ and $H$, denoted by $G\times H$, is the graph with vertex set $\{(v_1,\ v_2):\ v_1\in G,\ v_2\in H\}$ where two vertices $(v_1,\ v_2)$ and $(v_1',\ v_2')$ are adjacent if and only if $v_1v_1'\in E(G)$ and $v_2v_2'\in E(H)$. 

This is because we usually only study the isoperimetric inequalities on connected graphs, but $\times_{i=1}^n P_{\N_0}$ and $\times_{i=1}^n P_{m_i}$ with $m_i\ge 2$ are not connected. The fact that $P_{m_1}\times P_{m_2}$ is not connected was also mentioned in \cite{HT}. Here we present a more general proposition.

\begin{proposition}
    For any $n\ge 2$ and $m_1,\ m_2,\ ...,\ m_n\ge 2$, each of $\times_{i=1}^n P_{\N_0}$ and $\times_{i=1}^n P_{m_i}$ has at least $2^{n-1}$ components.
\end{proposition}

\begin{proof}
    As $\times_{i=1}^n P_{m_i}$ is a subgraph of $\times_{i=1}^n P_{\N_0}$, we only need to prove this proposition for $\times_{i=1}^n P_{\N_0}$. If two vertices $(u_1,\ u_2,\ ...,\ u_n)$ and $(v_1,\ v_2,\ ...,\ v_n)$ in $\times_{i=1}^n P_{\N_0}$ are adjacent, then for any $1\le i\le n$, we have $|u_i-v_i|=1$, which means $u_i$ and $v_i$ have different parity. So if two vertices $(a_1,\ a_2,\ ...,\ a_n)$ and $(b_1,\ b_2,\ ...,\ b_n)$ are connected, then either $a_i$ and $b_i$ have the same parity for any $1\le i\le n$, or $a_i$ and $b_i$ have different parity for any $1\le i\le n$. So any two vertices in $\{(c_1,\ c_2,\ ...,\ c_{n-1},\ 0):\ c_j\in \{0,\ 1\}\ \forall\ 1\le j\le n-1\}$ are not connected, and we have at least $2^{n-1}$ components.
\end{proof}

For future goals, we can try extending the result in Theorem \ref{finite} to $\boxtimes_{i=1}^n P_{m_i}$.

\begin{problem}
    Given $n,\ k,\ m_1,\ m_2,\ ...,\ m_n\in \N$ with $n\ge 3$ and $1\le k\le \prod_{i=1}^n m_i$, determine $\min_{|S|=k}|\partial_{\boxtimes_{i=1}^n P_{m_i}}(S)|$.
\end{problem}

The first case we can consider is the symmetric $3$-dimensional case.

\begin{problem}
    Given $k,\ m\in \N$ with $1\le k\le m^3$, determine $\min_{|S|=k}|\partial_{P_m\boxtimes P_m\boxtimes P_m}(S)|$.
\end{problem}

\section*{Acknowledgments}
I thank Dmitry Tsarev for bringing reference \cite{VR} to my attention.

\section*{Data availability}
There are no data associated with this paper.

\section*{Declarations}
\textbf{Competing Interests:} The author declares that he has no competing interests.

\end{document}